\documentclass[a4paper,11pt]{amsart}
\usepackage{mathptmx}

\title{Ideal-adic semi-continuity of minimal log discrepancies on surfaces}

\author{Masayuki Kawakita}
\address{Research Institute for Mathematical Sciences, Kyoto University, Kyoto 606-8502, Japan}
\email{masayuki@kurims.kyoto-u.ac.jp}

\theoremstyle{plain}
\newtheorem{theorem}{Theorem}
\newtheorem{lemma}[theorem]{Lemma}
\newtheorem{conjecture}[theorem]{Conjecture}

\theoremstyle{definition}
\newtheorem{definition}[theorem]{Definition}

\theoremstyle{remark}
\newtheorem{remark}[theorem]{Remark}
\newtheorem{notation}[theorem]{Notation}
\newtheorem*{acknowledgements}{Acknowledgements}

\newcommand{\bA}{\mathbb{A}}
\newcommand{\bN}{\mathbb{N}}
\newcommand{\bQ}{\mathbb{Q}}
\newcommand{\bR}{\mathbb{R}}
\newcommand{\cD}{\mathcal{D}}
\newcommand{\cI}{\mathcal{I}}
\newcommand{\cO}{\mathcal{O}}
\newcommand{\fa}{\mathfrak{a}}
\newcommand{\fb}{\mathfrak{b}}
\newcommand{\fm}{\mathfrak{m}}

\DeclareMathOperator{\mld}{mld}
\DeclareMathOperator{\mult}{mult}
\DeclareMathOperator{\ord}{ord}
\DeclareMathOperator{\Supp}{Supp}

\begin{document}
\begin{abstract}
We prove the ideal-adic semi-continuity of minimal log discrepancies on surfaces.
\end{abstract}

\maketitle

De Fernex, Ein and Musta\c{t}\u{a} in \cite{dFEM10} after Koll\'ar in \cite{Kl08} proved the ideal-adic semi-continuity of log canonicity effectively, to obtain Shokurov's ACC conjecture \cite{S96} for log canonical thresholds on smooth varieties. Musta\c{t}\u{a} formulated this semi-continuity for minimal log discrepancies.

\begin{conjecture}[Musta\c{t}\u{a}, see \cite{K10}]\label{cnj:mld}
Let $(X,\Delta)$ be a pair, $Z$ a closed subset of $X$ and $\cI_Z$ its ideal sheaf. Let $\fa=\prod_{j=1}^k\fa_j^{r_j}$ be a formal product of ideal sheaves $\fa_j$ with positive real exponents $r_j$. Then there exists an integer $l$ such that the following holds\textup{:} if $\fb=\prod_{j=1}^k\fb_j^{r_j}$ satisfies $\fa_j+\cI_Z^l=\fb_j+\cI_Z^l$ for all $j$, then
\begin{align*}
\mld_Z(X,\Delta,\fa)=\mld_Z(X,\Delta,\fb).
\end{align*}
\end{conjecture}

The case of minimal log discrepancy zero is the semi-continuity of log canonicity. Conjecture \ref{cnj:mld} is proved in the Kawamata log terminal (klt) case in \cite[Theorem 1.6]{K10}. It is however inevitable to treat log canonical (lc) singularities in the study of limits of singularities; for example, the limit of klt pairs $(\bA^2_{x,y},(x,y^n)\cO_{\bA^2})$ indexed by $n\in\bN$ is the lc pair $(\bA^2,x\cO_{\bA^2})$. The purpose of this paper is to settle Musta\c{t}\u{a}'s conjecture for surfaces.

\begin{theorem}\label{thm:surface}
Conjecture \ref{cnj:mld} holds when $X$ is a surface.
\end{theorem}

We must handle a non-klt triple $(X,\Delta,\fa)$  which has positive minimal log discrepancy, but unlike the klt case, the log canonicity is no longer retained once when $\fa$ is expanded. However for surfaces, we are reduced to the purely log terminal (plt) case in which $\fa$ has an expression $\fa'\cO_X(-C)$, then we can expand only the part $\fa'$ to apply the result on log canonicity.

We work over an algebraically closed field of characteristic zero. We use the notation below for singularities in the minimal model program.

\begin{notation}
A \textit{pair} $(X,\Delta)$ consists of a normal variety $X$ and an effective $\bR$-divisor $\Delta$ such that $K_X+\Delta$ is an $\bR$-Cartier $\bR$-divisor. We treat a \textit{triple} $(X,\Delta,\fa)$ by attaching a formal product $\fa=\prod_j\fa_j^{r_j}$ of finitely many coherent ideal sheaves $\fa_j$ with positive real exponents $r_j$. A prime divisor $E$ on a normal variety $X'$ with a proper birational morphism $\varphi\colon X'\to X$ is called a divisor \textit{over} $X$, and the image $\varphi(E)$ on $X$ is called the \textit{centre} of $E$ on $X$ and denoted by $c_X(E)$. We denote by $\cD_X$ the set of divisors over $X$. The \textit{log discrepancy} $a_E(X,\Delta,\fa)$ of $E$ is defined as $1+\ord_E(K_{X'}-\varphi^*(K_X+\Delta))-\ord_E\fa$. The triple $(X,\Delta,\fa)$ is said to be \textit{log canonical}, \textit{Kawamata log terminal} if $a_E(X,\Delta,\fa)\ge0$, $>0$ respectively for all $E\in\cD_X$, and said to be \textit{purely log terminal}, \textit{canonical}, \textit{terminal} if $a_E(X,\Delta,\fa)>0$, $\ge1$, $>1$ respectively for all exceptional $E\in\cD_X$. A centre $c_X(E)$ with $a_E(X,\Delta,\fa)\le0$ is called a \textit{non-klt centre}. Let $Z$ be a closed subset of $X$. The \textit{minimal log discrepancy} $\mld_Z(X,\Delta,\fa)$ over $Z$ is the infimum of $a_E(X,\Delta,\fa)$ for all $E\in\cD_X$ with centre in $Z$. We say that $E\in\cD_X$ \textit{computes} $\mld_Z(X,\Delta,\fa)$ if $c_X(E)\subset Z$ and $a_E(X,\Delta,\fa)=\mld_Z(X,\Delta,\fa)$ (or negative when $\mld_Z(X,\Delta,\fa)=-\infty$).
\end{notation}

Prior to the proof of Theorem \ref{thm:surface}, we collect standard reductions and known results on Conjecture \ref{cnj:mld}.

\begin{lemma}[{\cite[Remarks 1.5.3, 1.5.4]{K10}}]\label{lem:standard}
Conjecture \ref{cnj:mld} is reduced to the case when $X$ has $\bQ$-factorial terminal singularities, $\Delta=0$ and $Z$ is irreducible\textup{;} and it suffices to prove the inequality $\mld_Z(X,\fa)\le\mld_Z(X,\fb)$.
\end{lemma}

\begin{theorem}\label{thm:collection}
Conjecture \ref{cnj:mld} holds in each of the following cases.
\begin{enumerate}
\item\label{itm:coll.known.nonlc}
$\mld_Z(X,\fa)=-\infty$.
\item\label{itm:coll.known.lc}
\textup{(Koll\'ar \cite{Kl08}, de Fernex, Ein, Musta\c{t}\u{a} \cite{dFEM10})} $\mld_Z(X,\fa)=0$.
\item\label{itm:coll.known.klt}
\textup{(\cite[Theorem 1.6]{K10})} $(X,\fa)$ is klt about $Z$.
\end{enumerate}
\end{theorem}

\begin{remark}\label{rmk:effective}
In (\ref{itm:coll.known.lc}) above, one can take as $l$ any integer greater than the maximum of $\ord_E\fa_j/\ord_E\cI_Z$, by fixing $E\in\cD_X$ which computes $\mld_Z(X,\fa)$. The estimate of $l$ in (\ref{itm:coll.known.klt}) involves the log canonical threshold of $\fa$.
\end{remark}

Conjecture \ref{cnj:mld} for surfaces is reduced to the plt case.

\begin{lemma}\label{lem:reduction}
One may assume the following for Conjecture \ref{cnj:mld} for surfaces.
\begin{enumerate}
\item
$X$ is a smooth surface, $\Delta=0$ and $Z$ is a closed point.
\item
$(X,\fa)$ is plt with unique non-klt centre $C$.
\item
$C$ is a smooth curve.
\end{enumerate}
\end{lemma}

\begin{proof}
We may assume that $X$ is smooth with $\Delta=0$ by Lemma \ref{lem:standard}, and may assume $\mld_Z(X,\fa)>0$ by Theorem \ref{thm:collection}(\ref{itm:coll.known.nonlc}), (\ref{itm:coll.known.lc}). Let $C$ be the non-klt locus of $(X,\fa)$. By Theorem \ref{thm:collection}(\ref{itm:coll.known.klt}), we have only to work about $Z\cap C$. The assumption $\mld_Z(X,\fa)>0$ means that $Z$ contains no non-klt centre, whence $Z\cap C$ consists of finitely many closed points. By replacing $Z$ with $Z\cap C$ and working locally, we may assume that $Z$ is a closed point $x$, and $(X,\fa)$ has the non-klt locus $C$ which is a curve. The exceptional divisor $E$ of the blow-up of $X$ at $x$ has positive log discrepancy $a_E(X,\fa)$, but it is at most $a_E(X,C)=2-\mult_xC$. So $C$ must be smooth at $x$.
\end{proof}

We work locally about the closed point $x=Z$ with the assumptions in Lemma \ref{lem:reduction}. We denote by $\fm$ the maximal ideal sheaf at $x$, and use the notation similar to \cite[Definition 1.3]{K10}.

\begin{definition}
For $\fb=\prod_j\fb_j^{r_j}$ and $l\in\bN$, we write $\fa\equiv_l\fb$ if $\fa_j+\fm^l=\fb_j+\fm^l$ for all $j$.
\end{definition}

Set $c:=\mld_x(X,\fa)$. The non-trivial locus of $\fa$ is a divisor of form $C+D$. Since $(X,\fa)$ is plt, we can fix $s,t>0$ and $t'\ge0$ such that $\mld_x(X,sD,\fa\fm^{t'})=\mld_x(X,\fa\fm^t)=0$. We fix a log resolution $\varphi\colon\bar{X}\to X$ of $(X,\fa\fm)$, that is, $\prod_j\fa_j\fm\cO_{\bar{X}}$ defines a divisor with simple normal crossing support. Let $\bar{C},\bar{D}$ denote the strict transform of $C,D$. Since $C$ is smooth, $\bar{C}$ intersects only one prime divisor $F$ in $\varphi^{-1}(x)$. This will play a crucial role in the proof. By blowing up $\bar{X}$ further, we may assume that every divisor $E$ in $\varphi^{-1}(x)$ intersecting $\bar{D}$ satisfies
\begin{align}\label{eqn:ord.D}
\ord_ED\ge s^{-1}c-1.
\end{align}
We take an integer $l$ such that
\begin{align}\label{eqn:ord.a}
l>\ord_E\fa_j/\ord_E\fm
\end{align}
for all $j$ and $E\subset\varphi^{-1}(x)$. The lemma below is an application of Theorem \ref{thm:collection}(\ref{itm:coll.known.lc}) and Remark \ref{rmk:effective}, with the inequality (\ref{eqn:ord.a}).

\begin{lemma}\label{lem:lc}
$\mld_x(X,sD,\fb\fm^{t'})=\mld_x(X,\fb\fm^t)=0$ for any $\fb\equiv_l\fa$.
\end{lemma}

We write
\begin{align*}
\fa_j\cO_{\bar{X}}=\cO_{\bar{X}}(-H_j-V_j)
\end{align*}
with divisors $H_j,V_j$ such that $\Supp H_j\subset\bar{C}+\bar{D}$ and $\Supp V_j\subset\varphi^{-1}(x)$. Let $\fb\equiv_l\fa$. For $E\subset\varphi^{-1}(x)$, we have $\ord_E\fa_j<\ord_E\fm^l$ by (\ref{eqn:ord.a}), and $\ord_E\fa_j=\ord_E\fb_j$ by $\fa_j+\fm^l=\fb_j+\fm^l$. Hence we can write
\begin{align*}
\fb_j\cO_{\bar{X}}=\fb'_j\cO_{\bar{X}}(-V_j),\qquad\fm^l\cO_{\bar{X}}=\cO_{\bar{X}}(-M_j-V_j),
\end{align*}
with an ideal sheaf $\fb'_j$ and an effective divisor $M_j$ such that $\Supp M_j=\varphi^{-1}(x)$. Then the equality $\fa_j+\fm^l=\fb_j+\fm^l$ induces
\begin{align}\label{eqn:support}
\cO_{\bar{X}}(-H_j)+\cO_{\bar{X}}(-M_j)=\fb'_j+\cO_{\bar{X}}(-M_j).
\end{align}

The following lemma shows $\mld_x(X,\fb)\ge c$, which with Lemma \ref{lem:standard} completes Theorem \ref{thm:surface}.

\begin{lemma}\label{lem:centre}
$a_G(X,\fb)\ge c$ for any $\fb\equiv_l\fa$ and $G\in\cD_X$ with $c_X(G)=x$.
\end{lemma}

\begin{proof}
We divide into three cases according to the position of $c_{\bar{X}}(G)$.
\begin{enumerate}
\item\label{itm:notCD}
$c_{\bar{X}}(G)\not\subset\bar{C}+\bar{D}$.
\item\label{itm:D}
$c_{\bar{X}}(G)\subset\bar{D}$.
\item\label{itm:C}
$c_{\bar{X}}(G)\subset\bar{C}$.
\end{enumerate}

(\ref{itm:notCD})
By (\ref{eqn:support}), $\Supp H_j\cap\Supp M_j=\Supp\cO_{\bar{X}}/\fb'_j\cap\Supp M_j$, whence $\Supp\cO_{\bar{X}}/\fb'_j\cap\varphi^{-1}(x)\subset\bar{C}+\bar{D}$. In particular, $c_{\bar{X}}(G)\not\subset\Supp\cO_{\bar{X}}/\fb'_j$. This implies $\ord_G\fb_j=\ord_GV_j=\ord_G\fa_j$, so $a_G(X,\fb)=a_G(X,\fa)\ge c$.

(\ref{itm:D})
Take a prime divisor $E$ in $\varphi^{-1}(x)$ such that $c_{\bar{X}}(G)\subset E$. By (\ref{eqn:ord.D}), $\ord_GD=\ord_ED\cdot\ord_GE+\ord_G\bar{D}\ge\ord_ED+1\ge s^{-1}c$. Lemma \ref{lem:lc} for $(X,sD,\fb\fm^{t'})$ implies $a_G(X,\fb)\ge s\ord_GD$. These two inequalities induce $a_G(X,\fb)\ge c$.

(\ref{itm:C})
$c_{\bar{X}}(G)$ is in the unique divisor $F\subset\varphi^{-1}(x)$ intersecting $\bar{C}$. There exists a divisor $E$ in $\varphi^{-1}(x)$ with $a_E(X,\fa\fm^t)=0$. Let $L$ be the union of all such $E$. Then $L\cup\bar{C}$ is connected by the connectedness lemma \cite[Theorem 17.4]{Kl+92}. Hence $F\subset L$, that is, $a_F(X,\fa\fm^t)=0$, so $\ord_F\fm^t=a_F(X,\fa)\ge c$ (actually $=c$ by precise inversion of adjunction \cite{EMY03}). Lemma \ref{lem:lc} for $(X,\fb\fm^t)$ implies $a_G(X,\fb)\ge\ord_G\fm^t$. With $c_{\bar{X}}(G)\subset F$, we obtain $a_G(X,\fb)\ge\ord_G\fm^t\ge\ord_F\fm^t\ge c$.
\end{proof}

\begin{remark}
The case division in the proof of Lemma \ref{lem:centre} is in terms of the union $H$ of divisors $E$ with $\ord_E\fa>0$ and $c_X(E)\not\subset Z$, on a suitable log resolution $\bar{X}$. We write $H=H'+H''$ so that $H'$ is the union of those $E$ with $a_E(X,\fa)=0$. Then the cases (\ref{itm:notCD}), (\ref{itm:D}), (\ref{itm:C}) correspond to the conditions (\ref{itm:notCD}) $c_{\bar{X}}(G)\not\subset H$, (\ref{itm:D}) $\subset H''$ and $\not\subset H'$, (\ref{itm:C}) $\subset H'$ respectively. The proof of (\ref{itm:notCD}) works in any dimension, and (\ref{itm:D}) works as long as $(X,\fa)$ is plt (or more generally, dlt). However, (\ref{itm:C}) would not work unless $H'$ intersects only one divisor in $\varphi^{-1}(Z)$.
\end{remark}

\begin{remark}
In \cite{K10}, Conjecture \ref{cnj:mld} is formulated for $(X,\Delta,\fa)$ with $\fa$ an $\bR$-ideal sheaf as an equivalence class of formal products of ideal sheaves. Our proof is valid also for this formulation.
\end{remark}

{\small
\begin{acknowledgements}
This paper was generated in the discussions during the workshop at American Institute of Mathematics. I am grateful to Professor T. de Fernex for his suggestion of the connectedness lemma after increasing the boundary. I thank Mr Y. Nakamura for his interest in the surface case and Professor M. Musta\c{t}\u{a} for his conjecture. American Institute of Mathematics supported my participation financially. The research was partially supported by Grant-in-Aid for Young Scientists (A) 24684003.
\end{acknowledgements}
}

\end{document}